\documentclass[3p,12pt,review]{elsarticle}

\usepackage{amssymb}
 \usepackage{amsmath}
 \usepackage{amsthm}
 \usepackage{hyperref}
 \usepackage{algorithmic}
 \usepackage{algorithm}
 \usepackage{ctable}
 \usepackage{multirow}
 \usepackage{cases}

 \newtheorem{thm}{Theorem}[section]

 \newtheorem{cor}[thm]{Corollary}
 \newtheorem{lem}[thm]{Lemma}
 \newtheorem{prop}[thm]{Proposition}
 \theoremstyle{definition}
 \newtheorem{defn}[thm]{Definition}
 \theoremstyle{remark}
 \newtheorem{rem}[thm]{Remark}
 \theoremstyle{definition}
 \newtheorem{exm}{Example}[section]
 \DeclareMathOperator{\prev}{prev}
  \DeclareMathOperator{\new}{new}

 \def\Blem{\begin{lem}}
 \def\Elem{\end{lem}}
 \newcommand{\uccomment}[1]{}
 \def\Bpr{\begin{prop}}
 \def\Epr{\end{prop}}

\usepackage[multiple]{footmisc}
\usepackage[para]{manyfoot}
\DeclareNewFootnote{B}[alph]

\newcounter{cCounter}
 \def\Bp{\begin{proof}}
 \def\Ep{\end{proof}}

 \def\Bex{\begin{exm}}
 \def\Eex{\end{exm}}

 \def\Bcor{\begin{cor}}
 \def\Ecor{\end{cor}}

 \def\Br{\begin{rem}}
 \def\Er{\end{rem}}
 \def\Bthm{\begin{thm}}
 \def\Ethm{\end{thm}}
 \def\Bd{\begin{defn}}
 \def\Ed{\end{defn}}
 \def\Beq{\begin{equation}}
 \def\Eeq{\end{equation}}

 \DeclareMathOperator{\st}{subject\; to}
 \DeclareMathOperator{\minimize}{minimize}

 \def\RR{\mathbb{R}}
 
\def\bone{{\boldsymbol 1}}

\journal{...}

\begin{document}

\begin{frontmatter}

\title{A Two-Phase Method for Solving Continuous Rank-One Quadratic Knapsack Problems}

\author[mathalz]{Sayyed Ehsan Monabbati }
\ead{e.monabbati@alzahra.ac.ir}


\address[mathalz]{Department of Mathematics, Faculty of Mathematical Sciences, Alzahra University, Tehran, Iran}
\cortext[cor1]{Corresponding author, e.monabbati@alzahra.ac.ir, Tel: +98-21-85693001}

\begin{abstract}
In this paper, we propose a two-phase algorithm for solving continuous rank-one quadratic knapsack problems (R1QKP).
In particular, we study the solution structure of the problem without the knapsack constraint.
We propose an $O(n\log n)$ algorithm in this case.
We then use the solution structure to propose an $O(n^2\log n)$ algorithm that finds an interval containing the optimal value of the Lagrangian dual of R1QKP.
In the second phase, we solve the restricted Lagrangian dual problem using a traditional single-variable optimization method.
We perform a computational test on random instances and compare our algorithm with the general solver CPLEX.
\end{abstract}


\end{frontmatter}
\section{Introduction}\label{sec:intro}
The quadratic knapsack problem (QKP) deals with minimizing a quadratic function over one allocation constraint together with simple bounds on decision variables.
Formally, this problem can be written as to
\begin{subequations}\label{eqn:QKP-general}
\begin{align}
  \minimize \quad & \frac{1}{2} x^\top  Q x - c^\top x, \\
  \st \quad  & a^\top x = b, \label{eqn:QKP-gneral-eq}\\
  & 0 \leq x\leq u,
\end{align}
\end{subequations}
where $Q$ is a symmetric $n\times n$ matrix, $a,c,l,u\in\RR^n$ and $b\in \RR$.
QKP as a quadratic optimization problem is polynomially solvable when $Q$ is positive definite matrix \cite{pardalos1991algorithms}.
When $Q$ is diagonal with strictly positive diagonal entries QKP can be viewed as a strictly convex separable optimization problem that has many applications (e.g. resource allocation \cite{patriksson2015algorithms, patriksson2008survey, bitran1981disaggregation} and multicommodity network flows \cite{helgason1980polynomially}).
The solution methods for solving this type of QKPs usually rely on the fact that the optimal solution of the Lagrangian dual subproblems can be explicitly obtained in terms of the Lagrange multiplier $\lambda$ of \eqref{eqn:QKP-gneral-eq}.
Therefore the problem reduces to find a value for $\lambda$ such that the solution of the corresponding Lagrangian subproblem is satisfied in equality constraint \eqref{eqn:QKP-gneral-eq}.
The resulting equation is solved by different methods.
Helgason et. al \cite{helgason1980polynomially} propose an $O(n\log n)$  algorithm for solving the equation based on searching breakpoints of the Lagrangian dual problem.
Brucker  \cite{brucker1984n} finds an $O(n)$ bisection algorithm based on the properties of the Lagrangian dual function.
Dai and Fletcher \cite{dai2006new} propose a two-phase method; A bracketing phase that determines an interval containing the solution followed by the secant phase that approximates the solution within the promising interval.
This method is modified by Comminetti  et. al \cite{cominetti2014newton} with ignoring the bracketing phase and using a semi-smooth Newton method instead of the secant method.
Liu and Yong-Jin \cite{liu2017fast} consider a special case of the strictly convex form of the problem. They find the solution structure of the subproblems and use it in a modified secant algorithm.

Robinson et. al  \cite{robinson1992continuous} use the geometric interpretation of the problem and propose an algorithm that works in the primal space rather than the dual space.
This algorithm iteratively fixes variables and terminates after at most $n$ iterations.

In a more general case, when $Q$ is positive semidefinite in \eqref{eqn:QKP-general},
Dussault et. al \cite{dussault1986convex} propose an iterative algorithm in which a QKP with diagonal $Q$ should be solved in each iteration.
Paradalos et. al \cite{pardalos1991algorithms} suggest a potential reduction algorithm to solve this class of QKP.
di Serafino et. al \cite{di2018two} propose a two-phase gradient projection that has acceptable numerical performance in comparison with similar gradient-based methods.

QKPs with positive definite $Q$ are also solved by a gradient projection method \cite{dai2006new}, and an augmented-Lagrangian approach \cite{fletcher2017augmented}.

In this paper, we suppose $Q$ is a rank one symmetric matrix, that is $Q=qq^\top $ for some $q\in \RR^n$.
Moreover, we assume that  $0 <u$.
Without loss of generality we assume that $q_i\neq 0$ for each $i$.
By the following change of variables
\[
\begin{aligned}
   x_i \leftarrow q_i x_i, \quad c_i \leftarrow \frac{c_i}{q_i}, \quad a_i \leftarrow \frac{a_i}{q_i}, \quad l_i \leftarrow \min\{0, q_i u_i\},\quad u_i \leftarrow \max\{0, q_i u_i\},
\end{aligned}
\]
problem \eqref{eqn:QKP-general} is reduced to
\begin{subequations}\label{eqn:r1-qkp-red}
\begin{align}
  \minimize \quad & \frac{1}{2} (\bone^\top x)^2 - c^\top x, \\
  \st \quad  & a^\top x = b, \label{eqn:r1-qkp-red-eq}\\
  & 0\leq x\leq u.
\end{align}
\end{subequations}

Sharkey et. al \cite{sharkey2011class} study a class of nonseparable nonlinear knapsack problems in which one has to
\begin{equation}\label{eqn:nonsep-KP}
\begin{aligned}
  \minimize \quad &  g(s^\top x) - c^\top x, \\
  \st \quad  & a^\top x = b, \\
  & l\leq x\leq u,
\end{aligned}
\end{equation}
where $g:\RR\to\RR$ is an arbitrary real-valued function, and $s\in \RR^n$ is given.
They introduce an algorithm for solving \eqref{eqn:nonsep-KP} that runs in $O(n^2 \max\{\log n,\phi\})$, where $\phi$ is the time required to solve a single-variable optimization problem $\min\{ g(S) - \alpha S \;:\; L\leq S\leq U\}$ for given $\alpha, L, U\in \RR$.
With $g(t) = t^2$ and $s$ equals to the all-one vector, problem \eqref{eqn:r1-qkp-red} is a special case of problem \eqref{eqn:nonsep-KP}. That is, there exists an $O(n^2 \max\{\log n,1\})=O(n^2 \log n)$ algorithm for solving problem \eqref{eqn:r1-qkp-red}.

In this paper, we propose a two-phase algorithm for problem \eqref{eqn:r1-qkp-red}.
In section~\ref{sec:bnd-ver} we study the solution structure of the bounded version of the problem in which the equality constraint \eqref{eqn:r1-qkp-red-eq} is omitted.
We show that the bounded version could be solved in $O(n\log n)$ time.
In section~\ref{sec:the-algorithm}, in phase I, we use the solution structure of the bounded version to find an interval that may contain the optimal value of the Lagrangian dual function.
This is done in $O(n^2\log n)$ time in the worst case.
Then, we perform phase II, in which we explore the interval by a single-variable optimization method to find the optimal Lagrangian multiplier with desired precision.
In section~\ref{sec:comp-test}, we perform a computational test.
In particular, we compare the algorithm with the general quadratic programming solver CPLEX.

%

\section{Solution structure of the bounded version}\label{sec:bnd-ver}
In this section we consider the following bounded version of the problem
\begin{subequations}\label{eqn:r1-QKP-unres}
\begin{align}
  \minimize \quad & f(x) = \frac{1}{2} (\bone^\top x)^2 - c^\top x, \\
  \st \quad  & 0\leq x\leq u.
\end{align}
\end{subequations}
We propose a characterization of the solution in the primal space.
Note that most of algorithms for such problems use the so-called KKT conditions to study the solution structure.
Without loss of generality, we  assume that $c_1 \geq c_2 \geq \cdots\geq c_n \geq 0$, and
$l_i = 0$, $i=1,\cdots,n$.
Given two vectors $a,b\in\RR^n$, we denote the set $\{x \;:\; a\leq x\leq b\}$ by $[a,b]$.
Finally, given vector $u\in \RR^n$ we define $U_k = \sum_{i=1}^k u_i$ for $k=1,\cdots, n$, and $U_0=0$.

We begin with   preliminary lemmas.
\begin{lem} \label{lem:pre-simple-integer}
For $k=1,\cdots,n$ define $x^{(k)}$ as
\begin{align*}
x_i^{(k)} =
\begin{cases}
 u_i,  & i=1,\dots,k, \\
 0, & i=k+1,\dots,n,
\end{cases}
\end{align*}
and $x^{(0)}$ as the all-zero vector, and, define $G_k$ as
\[
G_k =  \frac12\left( U_k + U_{k-1}\right)-c_{k} = U_{k-1} + \frac12 u_{k} - c_k.
\]
Then the following assertions hold
 \begin{enumerate}[(i)]
 \item\label{itm:pre_simple_integer_p1}
 If $\bar n$ is the smallest index  in  $ \{ 1,\cdots, n \}$ such that $G_{\bar n}\geq 0$
then $\min_{i=1,\cdots,n} f(x^{(i)})= f(x^{(\bar n-1)}) $.
\item\label{itm:pre-simple-integer-p2}
If $G_k<0$ for all $k= 1,\cdots, n-1$ then $\min_{i=1,\cdots,n} f(x^{(i)})=  f(x^{(n)})$.
\end{enumerate}
\end{lem}

\begin{proof} \eqref{itm:pre_simple_integer_p1}
For $1\leq k\leq n-1$ we have
\begin{align*}
G_k - G_{k+1}&  = \frac12\left( U_k + U_{k-1}\right) - c_k -\frac12\left( U_{k+1} + U_k \right)+c_{k+1}
   \\
& = -\frac12\left(u_{k}+u_{k+1} \right) + (c_{k+1}-c_{k})\\
&  < 0\cdot
\end{align*}
Thus
\[
G_1 < G_2 < \cdots < G_{\bar n-1} < 0 \leq G_{\bar n} < G_{\bar n+1}  < \cdots < G_n.
\]
On the other hand, for $1\leq k \leq n-1$ we have
\begin{equation}\label{eqn:fxdiff}
\begin{aligned}
f(x^{(k+1)})-f(x^{(k)}) & = \frac12 U_{k+1}^2 - \sum_{i=1}^{k+1} u_i c_i - \frac12 U_k^2 +   \sum_{i=1}^{k} u_i c_i  \\
& =  \frac12 U_{k+1}^2 - u_{k+1}c_{k+1}  - \frac12 U_k^2  \\
& =  \frac12 (U_{k+1}^2- U_k^2) - u_{k+1}c_{k+1} \\
& =  \frac12 (U_{k+1}- U_k)(U_{k+1}+ U_k) - u_{k+1}c_{k+1}    \\
& = u_{k+1} \left( \frac12 (U_{k+1}+U_k) - c_{k+1} \right) \\
& = u_{k+1} G_{k+1}.
\end{aligned}
\end{equation}
Now let $m>\bar n-1$. Then
\[
\begin{aligned}
  f(x^{(m)}) - f(x^{(\bar n-1)}) & =  f(x^{(m)})-f(x^{(m-1)}) + f(x^{(m-1)}) + \cdots   + f(x^{(\bar n)}) - f(x^{(\bar n-1)}) \\
  & = u_mG_m + \cdots + u_{\bar n} G_{\bar n} > G_{\bar n}(u_m + \cdots + u_{\bar n +1}) \\
  & > 0.
\end{aligned}
\]
Similarly, if $m<\bar n-1$, then we have $f(x^{(m)}) - f(x^{(\bar n-1)}) >0$.

\eqref{itm:pre-simple-integer-p2}
The second part can be easily proved considering \eqref{eqn:fxdiff}.
\end{proof}

We need the following result about two-dimensional version of problem \eqref{eqn:r1-QKP-unres}.
\begin{lem}\label{lem:2d-version}
  Consider the following optimization problem
    \begin{equation}\label{eqn:2d-version}
     \begin{aligned}
     \minimize \quad & f(x_1,x_2) = \frac12 (x_1+x_2)^2 - c_1 x_1 - c_2 x_2, \\
     \st \quad & 0 \leq x_1 \leq u_1, \\
     & 0\leq x_2 \leq u_2,
     \end{aligned}
    \end{equation}
  where $c_1\geq c_2 \geq 0$ and $u_1$ and $u_2$ are real constants.
  Define  set $I=I_1\cup I_2$ where $I_1=\{ (u_1,x_2) \;:\; 0\leq x_2\leq u_2 \}$ and $I_2 = \{ (x_1,0)\;:\; 0\leq x_1\leq u_1\}$. Then, problem \eqref{eqn:2d-version} has no optimal solution outside of $I$.

\end{lem}

\begin{proof}
  If $c_1 = c_2$, then $f(x_1,x_2) =  \frac12 (x_1+x_2)^2 - c_1( x_1 + x_2)=\frac12 z^2 - c_1 z = g(z)$ where $z=x_1+x_2$.
  It is easy to see that $x^\ast = (x_1^\ast,x_2^\ast)$ with $x_1^\ast = \min\{c_1, u_1\}$ and $x_2^\ast = \min\{ c_1-x_1^\ast, u_2\}$ is the optimal solution of the problem, and we have $x^\ast \in I$.
  Assume that $c_1\neq c_2$.
  The feasible region of \eqref{eqn:2d-version} is equal to  $I_1\cup I_2\cup I_3\cup I_4$ where $I_3=\{(x_1,x_2)\;:\; 0<x_1<u_1, 0<x_2<u_2\}$ and $I_4=\{ (0,x_2)\;:\; 0<x_2<u_2\}\cup \{ (x_1,u_2)\;:\; 0<x_1<u_1\}$. We show that there is no optimal solution in $I_3$ and $I_4$.
  Indeed, we write the KKT optimality conditions as follows
  \begin{align}
    & x_1 + x_2 - c_1 + \alpha_1 - \alpha_2 = 0, \label{eqn:kkt42d-1}\\
    & x_1 + x_2 - c_2 + \beta_1 - \beta_2 = 0, \label{eqn:kkt42d-2}\\
    & \alpha_1 (x_1-u_1)=0, \quad \alpha_2 x_1 = 0, \label{eqn:kkt42d-3}\\
    & \beta_1 (x_2-u_2)=0, \quad \beta_2 x_2 = 0, \label{eqn:kkt42d-4}\\
    & 0\leq x_1 \leq u_1, \\
    & 0 \leq x_2 \leq u_2, \\
    & \alpha_1, \alpha_2, \beta_1, \beta_2 \geq 0,
  \end{align}
  where $\alpha_i$ and $\beta_i$, $i=1,2$ are KKT multipliers corresponding to the bound constraints.
  If $(x_1,x_2)\in I_3$ then from \eqref{eqn:kkt42d-3} and \eqref{eqn:kkt42d-4} we  have $\alpha_1 = \alpha_2 = \beta_1 = \beta_2 = 0$. Substituting these values in \eqref{eqn:kkt42d-1} and \eqref{eqn:kkt42d-2} implies that $c_1=c_2$ which contradicts our assumption.
  On the other hand, if $(x_1,x_2)\in I_4$ and  $x_1=0$, then we have $\alpha_1=0$.
  Now, \eqref{eqn:kkt42d-1} implies that $x_2=c_1+\alpha_2<0$.
  Thus $\beta_2=0$.
  From \eqref{eqn:kkt42d-2} we have $x_2 = c_2 - \beta_1$.
  Therefore, $c_2 = c_1 + \alpha_2 + \beta_1 \geq c_1$.
  This contradicts our assumption on $c_i$s.
  That is, problem \eqref{eqn:2d-version} has no optimal solution with $x_1=0$.
  By similar argument one can conclude that \eqref{eqn:2d-version} has no solution with $x_2=u_2$. This completes the proof.
\end{proof}

\begin{thm} \label{lem:pre2-simple-integer}
Suppose $x^{(k)}$ and $G_k$, $k=1,\cdots,n$ and  $\bar n$ are defined as in Lemma~\ref{lem:pre-simple-integer}.
Then the following assertions hold
\begin{enumerate}[(i)]
\item \label{lem:pre2-simple-integer-p1}
    If $\bar n>1$, then define $\delta_1,\delta_2$ as $\delta_1 =\min\left\{ c_{\bar n-1}-U_{\bar n-2}, u_{\bar n -1}\right\}$ and $\delta_2  =\max\left\{  c_{\bar n}-U_{\bar n -1}, 0 \right\}$.
      Also, define $\bar x$, $\tilde x$ as
     \[
     \bar x = x^{(\bar n-2)} + \delta_1 e_{\bar n-1}, \quad  \tilde x = x^{(\bar n-1)} + \delta_2 e_{\bar n },
     \]
      where  $e_i$ is the $i$th column of the identity matrix of dimension $n$.
     Then $\min\{ f(\bar x), f(\tilde x)\}$ is the optimal value of the following optimization problem
    \begin{equation}\label{eqn:restricted-simple-integer}
     \begin{aligned}
     \minimize \quad & f(x), \\
     \st \quad & x^{(\bar n-2)} \leq x \leq x^{(\bar n)}.
     \end{aligned}
    \end{equation}
\item\label{lem:pre2-simple-integer-p2}
    If $\bar n=1$, then define
    $\delta = \min\{c_{1},u_1\}$.
     Also define  $\tilde x  = \delta e_{1}$.
     Then $f(\tilde x)$ is the optimal value of the following optimization problem
    \begin{equation*}
     \begin{aligned}
     \minimize \quad & f(x), \\
      \st \quad & x^{(0)} \leq x \leq x^{(1)}.
     \end{aligned}
    \end{equation*}
\item\label{lem:pre2-simple-integer-p3}
    if $G_k<0$ for all $k=1,\cdots,n$, then define
    $\delta = \min\{c_n-U_{n-1},u_n\}$.
    Also define $\tilde x = x^{(n-1)} + \delta e_n$.
     Then $f(\tilde x)$ is the optimal value of the following optimization problem
    \begin{equation*}
     \begin{aligned}
     \minimize \quad & f(x), \\
      \st \quad & x^{(n-1)} \leq x \leq x^{(n)}.
     \end{aligned}
    \end{equation*}
\end{enumerate}

\end{thm}
\begin{proof}
(\ref{lem:pre2-simple-integer-p1})
 By Lemma~\ref{lem:2d-version} we can partition the optimal solution set as  $I_1\cup I_2$ where $I_1=[x^{(\bar n-2)},x^{(\bar n-1)}]$ and $I_2=  [x^{(\bar n-1)},x^{(\bar n)}]$.
 We show that $f(\bar x) = \min\{ f(x): x\in I_1\}$ and $f(\tilde x) = \min\{ f(x) : x\in I_2 \}$.
Indeed, we use a simple technique of single-variable calculus.
Let $x\in I_1$, then $x=x(\delta)$,
for some $\delta\in[0,u_{\bar n-1}]$, where $x(\delta) = x^{(\bar n-2)} + \delta e_{\bar n-1}$.
Thus the problem $\min\{ f(x): x\in I_1\}$  reduces to $\min\{ f(x(\delta)): 0\leq \delta\leq u_{\bar n-1}\}$.
On the other hand, one can write
\[
f(x(\delta)) = \frac12 \left( U_{\bar n-2}  + \delta \right)^2 - \sum_{i=1}^{\bar n -2} c_i u_i - c_{\bar n-1} \delta.
\]
We have $df(x(\delta))/d\delta =U_{\bar n-2}  + \delta - c_{\bar n-1}$.
Thus $df(x(\delta))/d\delta=0$ only if $\delta=\delta' = c_{\bar n-1}- U_{\bar n-2}$.
Since $d^2f(x(\delta))/d\delta^2 >0$ and $ \delta'>\frac12 u_{\bar n-1}$ the optimal value is achieved at $\delta_1$.

To prove $f(\tilde x) = \min\{ f(x) : x\in I_2 \}$, by the same argument as the previous paragraph,
it suffices to solve single optimization problem $\min\{ f(x(\delta)): 0\leq \delta\leq u_{\bar n}\}$,
where $x(\delta) = x^{(\bar n-1)} + \delta e_{\bar n}$. It is easy to see that if
$\delta=\delta' = c_{\bar n}- U_{\bar n-1}$ then $df(x(\delta))/d\delta=0$.
Since  $\delta'\leq \frac12 u_{\bar n}$, by definition of $\bar n$, then $f(\tilde x)$ is the optimal value of $\min\{ f(x) : x\in I_2 \}$.

The proof of parts \eqref{lem:pre2-simple-integer-p2} and  \eqref{lem:pre2-simple-integer-p3} is similar.
\end{proof}

The following Corollary~\ref{cor:optimality-special-case} states simple conditions under which the optimal solution of the problem in Theorem~\ref{lem:pre2-simple-integer}(\ref{lem:pre2-simple-integer-p1}) is $\bar x$ or $\tilde x$.
\begin{cor} \label{cor:optimality-special-case}
  In Theorem~\ref{lem:pre2-simple-integer}(\ref{lem:pre2-simple-integer-p1}),
  \begin{enumerate}[(i)]
    \item if $\delta_1 = u_{\bar n-1}$ then $\min\{ f(\bar x), f(\tilde x)\}=f(\tilde x)$, and, \label{cor:cor-of-main-lem-1}
    \item if $\delta_2 = 0$ then $\min\{ f(\bar x), f(\tilde x)\}=f(\bar x)$. \label{cor:cor-of-main-lem-2}
  \end{enumerate}
\end{cor}
\begin{proof}
For brevity, we just prove part (\ref{cor:cor-of-main-lem-1}). The proof of the second part is similar. Under the assumption of part (\ref{cor:cor-of-main-lem-1}), we have
\[
f(\bar x) - f(\tilde x) = \frac12 U_{\bar n-1}^2 - \sum_{i=1}^{\bar n-1}c_i u_i - \frac12 c_{\bar n}^2 + \sum_{i=1}^{\bar n-1} c_i u_i + c_{\bar n} (c_{\bar n} - U_{\bar n-1} ) = \frac12 \left( U_{\bar n-1} - c_{\bar n} \right)^2\geq 0.
\]
\end{proof}

Theorem~\ref{lem:pre2-simple-integer} solves a restricted version of  problem \eqref{eqn:r1-qkp-red}.
In the following theorem, we show that the solution of the restricted version is the solution of the original problem.

\begin{thm}\label{thm:main-result}
Define $G_k$'s, $x^{(k)}$'s,  $\bar n$, $\bar x$ and $\tilde x$ as in Theorem~\ref{lem:pre2-simple-integer}. Then, the following assertions  hold
\begin{enumerate}[(i)]
\item\label{thm:main-result-p1}
If $1<\bar n \leq n$, then $\min\{ f(\bar x), f(\tilde x)\}$ is the optimal value of \eqref{eqn:r1-QKP-unres}, where $\tilde x$ and $\bar x$ are defined as in Theorem~\ref{lem:pre2-simple-integer}(\ref{lem:pre2-simple-integer-p1}).
\item\label{thm:main-result-p2}
If $\bar n=1$, then $f(\tilde x)$ is the optimal value of \eqref{eqn:r1-QKP-unres}, where $\tilde x$ is defined as in Theorem~\ref{lem:pre2-simple-integer}(\ref{lem:pre2-simple-integer-p2}).
\item\label{thm:main-result-p3}
  If $G_k<0$ for all $k=1,\cdots, n$, then $f(\tilde x)$ is the optimal solution of \eqref{eqn:r1-QKP-unres}, where  $\tilde x = x^{(n-1)} + \delta' e_n$, and $\delta' = \min\{c_n - U_{n-1},u_n \}$.
\end{enumerate}
\end{thm}
\begin{proof}
For two vectors $x,z\in \RR^n$ we have
\begin{equation}\label{eqn:difff-f}
\begin{aligned}
f(z)-f(x) & =  \frac12 ( \bone^\top z + \bone^\top x)( \bone^\top z - \bone^\top x) - c^\top(z-x).
\end{aligned}
\end{equation}
Let $x$ be a feasible solution of \eqref{eqn:r1-QKP-unres}.
If $x=u=x^{(n)}$, then from the definition of $\bar n$ we have  $f(x^{\bar n-1}) \leq f(x)$, and the result follows from Theorem~\ref{lem:pre2-simple-integer}.
Suppose $x\neq u$.
 We show there exists a specially structured feasible solution $x'$ that is better than $x$.
Indeed, let $k$ be such that
\[
U_k \leq \bone^\top x < U_{k+1}.
\]
 Define vector $x'$ by
\[
x'_i =
\begin{cases}
u_i, & i=1,\cdots,k, \\
\bone^\top x  - U_k, & i = k+1,\\
0, & i=k+2,\cdots,n.
\end{cases}
\]
Then, clearly $x'$ is feasible for \eqref{eqn:r1-QKP-unres} and
\begin{align*}
\bone^\top x'  =  \sum_{i=1}^k x'_i + x'_{k+1} + \sum_{i=k+2}^n x'_i
   = \sum_{i=1}^k u_i + \sum_{i=1}^n x_i  - \sum_{i=1}^k u_i =\bone^\top x.
\end{align*}
Moreover we have
\begin{align*}
c^\top x'  & = \sum_{i=1}^k u_ic_i + c_{k+1} x'_{k+1} = \sum_{i=1}^k u_ic_i +
 c_{k+1}\sum_{i=1}^n x_i  - c_{k+1}\sum_{i=1}^k u_i \\
 &  = \sum_{i=1}^k u_ic_i +  c_{k+1}\sum_{i=1}^k x_i + c_{k+1}\sum_{i=k+1}^n x_i - c_{k+1}\sum_{i=1}^k u_i \\
 & \geq \sum_{i=1}^k u_ic_i +  \sum_{i=1}^k (x_i-u_i)c_i + \sum_{i=k+1}^n x_ic_i  \tag{by monotonicity of $c_i$'s}\\
 & = c^\top x.
\end{align*}
Therefor, \eqref{eqn:difff-f} implies that $f(x') - f(x) = -c^\top(x'-x) \leq 0$, i.e. $f(x')\leq f(x)$.

\eqref{thm:main-result-p1}
Now,  we consider three cases for index $k$ introduced in the definition of $x'$: $k\geq \bar n$, $k< \bar n -2$  and  $k=\bar n-1, \bar n-2$.
In the latter case, we have $x^{(\bar n-2)}\leq x'\leq x^{(\bar n)}$, so the assertion is true by
Theorem~\ref{lem:pre2-simple-integer}, since
\[
\min\{f(\bar x),f(\tilde x)\}  = \min\{ f(x) \;:\; x\in [x^{(\bar n-2)}, x^{(\bar n)}] \}\leq f(x') \leq f(x).
\]
We show in both the other  cases there is a point in the set  $\{x^{(i)}\}_{i=1,\cdots,n}$ better than $x'$, that is $f(x^{(i)})\leq f(x')$ for some $i=1,\cdots,n$.
By  Lemma~\ref{lem:pre-simple-integer}, $f(x^{(\bar n-1)})\leq f(x^{(i)})$ and the result follows by
 Theorem~\ref{lem:pre2-simple-integer}.

First, let $k \geq \bar n$. Then we have
\begin{align*}
f(x^{(k)}) - f(x') & = \frac12 ( \bone^\top x^{(k)} - \bone^\top x')( \bone^\top x^{(k)} + \bone^\top x') - c^\top(x^{(k)}-x') \\
& = -\frac12  x'_{k+1}  (2U_k+ x'_{k+1}) + c_{k+1}x'_{k+1} \\
& = -x'_{k+1} \left( \frac12 (2U_k+ x'_{k+1}) - c_{k+1} \right).
\end{align*}
On the other hand, we have
\[
2U_k+ x'_{k+1} = 2\sum_{i=1}^{\bar n-1} u_i +  \sum_{i=\bar n}^{k} u_i + x'_{k+1} \geq U_{\bar n} + U_{\bar n-1}.
\]
Therefor
\[
\frac12 (2U_k+ x'_{k+1}) - c_{k+1} \geq \frac12 (U_{\bar n} + U_{\bar n-1}) - c_{\bar n} = G_{\bar n} \geq 0.
\]
Thus $f(x^{(k)}) \leq f(x')$.

Now, let $k < \bar n-2$. Then
\begin{align*}
f(x^{(k+1)}) - f(x') & = \frac12 ( \bone^\top x^{(k+1)} - \bone^\top x')( \bone^\top x^{(k+1)} + \bone^\top x') - c^\top(x^{(k+1)}-x') \\
& = \frac12 ( u_{k+1} - x'_{k+1} )( U_{k+1} +  U_k + x'_{k+1} ) - c_{k+1}(u_{k+1}-x'_{k+1}) \\
& = (u_{k+1}-x'_{k+1})\left(\frac12 ( 2U_{k} +  x'_{k+1} + u_{k+1} ) - c_{k+1}\right).
\end{align*}
On the other hand, we have
\[
2U_{k} +  x'_{k+1} + u_{k+1} \leq 2U_{k} +  2u_{k+1} \leq 2U_{k} +  2\sum_{i=k+1}^{\bar n-2} u_i + u_{\bar n - 1} = U_{\bar n -2} + U_{\bar n -1}.
\]
Thus
\[
\frac12 ( 2U_{k} +  x'_{k+1} + u_{k+1} ) - c_{k+1} \leq \frac12 (U_{\bar n -2} + U_{\bar n -1}) - c_{\bar n -1} = G_{\bar n -1} < 0.
\]
That is $f(x^{(k+1)}) < f(x')$.
Thus in both cases there exist a point $x^{(t)}$, say,  such that $f(x^{(t)})\leq f(x') \leq f(x)$. Now, by  Lemma~\ref{lem:pre-simple-integer}, $f(x^{(\bar n-1)})\leq f(x^{(t)})\leq f(x)$ and the result follows by
 Theorem~\ref{lem:pre2-simple-integer}.

 \emph{Proof of }\eqref{thm:main-result-p2} Consider the possible values of $k$ at  the beginning of the proof of part  \eqref{thm:main-result-p1}. Here, we just have $k\geq \bar n=1$. Now, similar argument for this case proves \eqref{thm:main-result-p2}.

 \emph{Proof of }\eqref{thm:main-result-p3} Again consider the possible values of $k$ at  the beginning of the proof of part  \eqref{thm:main-result-p1}.  Similar argument with case $k<\bar n-2$ with $\bar n=n+1$ proves part \eqref{thm:main-result-p3}.

\end{proof}

One can conclude the following result on the time needed to solve problem \eqref{eqn:r1-QKP-unres}.
\begin{thm}
  There exists an $O(n\log n)$ time algorithm for problem \eqref{eqn:r1-QKP-unres}.
\end{thm}

\begin{proof}

Once the index $\bar n$ is determined the solution can be determined in $O(n)$ time.
We need $O(n\log n)$ to sort vector $c$, $O(n)$ to compute vector $G$, and $O(\log n)$ to find index $\bar n$.
That is, problem \eqref{eqn:r1-QKP-unres} can be solved in $O(n\log n)$.
\end{proof}

%
%
%

\section{The algorithm}
\label{sec:the-algorithm}

Our algorithm consists of two phases: bounding the optimal solution, and, computing the optimal solution to the desired precision.
The bounding phase is based on the Lagrangian dual of \eqref{eqn:r1-qkp-red} and the solution structure of the bounded version described in section~\ref{sec:bnd-ver}.

\subsection{Lagrangian dual}\label{sec:lag-rel}
Let $\lambda$ be the Lagrange multiplier of equality constraint in \eqref{eqn:r1-qkp-red}. Then, the Lagrangian function is defined as
\begin{equation}\label{eqn:lag-dual}
  \begin{aligned}
\phi(\lambda) & = \min\left\{ \frac12 (\bone^\top x)^2 - c^\top x + \lambda(b - a^\top x) \;:\; 0\leq x \leq u \right\} \\
& = \lambda b + \min\left\{ \frac12 (\bone^\top x)^2 - (c + \lambda a)^\top x  \;:\; 0\leq x \leq u \right\}
\end{aligned}
\end{equation}

We have the following fact about the structure of the Lagrangian function $\phi$.

\begin{thm}
Given Lagrange multiplier $\lambda$, define $\bar n$ as in Theorem~\ref{thm:main-result}. If $\bar n>1$, then
we have
\begin{align}
\phi(\lambda) &  =  \lambda b + f_\lambda(x^{(\bar n -1)}),  \text{if } c_{\bar n}(\lambda) \leq U_{\bar n -1} \leq  c_{\bar n-1}(\lambda), \tag{Type A}\label{eqn:phi-type-I} \\
\phi(\lambda) &  =  \lambda b +  p_{\bar n}(\lambda),  \text{if } U_{\bar n -1} \leq    c_{\bar n}(\lambda), \tag{Type B}\label{eqn:phi-type-II}\\
\phi(\lambda) &  =  \lambda b +   p_{\bar n-1}(\lambda),  \text{if } U_{\bar n -1} \geq   c_{\bar n -1}(\lambda), \tag{Type C}\label{eqn:phi-type-III}
\end{align}
where $f_\lambda$ is the objective function of the optimization part of \eqref{eqn:lag-dual}, and,
\[
\begin{aligned}
p_{k}(\lambda) & = -\frac12 a_{k}^2 \lambda^2 - a^\top d_{k}\lambda + \frac12c_{k}^2 - c^\top d_{k},  \\
d_k & = x^{(k-1)} +(c_k-U_{k-1}) e_k.
\end{aligned}
\]

\end{thm}
\begin{proof}
  The proof is based on the four possible cases for $\delta_1$ and $\delta_2$ in Theorem~\ref{thm:main-result}.
   We just prove \eqref{eqn:phi-type-I} and, for brevity, we omit the remaining parts.

   Suppose that $c_{\bar n}(\lambda) \leq U_{\bar n -1} \leq  c_{\bar n-1}(\lambda)$. Then we have $c_{\bar n-1}(\lambda) - U_{\bar n-2} \geq u_{\bar n-1}$ and $c_{\bar n}(\lambda) - U_{\bar n-1} \leq 0$. Therefor the value of $\delta_1$ and $\delta_2$ in Theorem~\ref{thm:main-result} can be determined as
   \[
   \delta_1 = \min\{ c_{\bar n-1}(\lambda) - U_{\bar n-2}, u_{\bar n-1}\} = u_{\bar n-1}, \qquad \delta_2 = \max\{c_{\bar n}(\lambda) - U_{\bar n-1} ,0 \}=c_{\bar n}(\lambda) - U_{\bar n-1}.
   \]
   Thus we have $\bar x = x^{(\bar n -2)} + \delta_1 e_{\bar n-1} = x^{(\bar n-1)}$.
   Thus, by some simplifications we have $f_\lambda(\tilde x) -  f_\lambda(\bar x)  = \frac12(c_{\bar n}(\lambda)-U_{\bar n-1})^2 \geq 0$. Now, by Theorem~\ref{thm:main-result}, we have $\min\{ f_\lambda(x)\;:\; 0\leq x\leq u\} = \min\{f(\tilde x),f(\bar x\} = f_\lambda(\bar x) = f_\lambda(x^{(\bar n-1)})$.
\end{proof}
Now, one may conclude that if $\bar n$ is fixed on an interval $[\lambda_a,\lambda_b]$, then $\phi(\lambda)$ is a piecewise function that contains exactly $3$ pieces.
However, the following simple example shows that this is not true.
\begin{exm}\label{exm:phi-bad-behavior}
Consider problem \eqref{eqn:r1-qkp-red} with the following parameters
\[
\begin{aligned}
& a^\top = \begin{bmatrix}
           -7 & -5 & 7 & -5 & 7
         \end{bmatrix}, \quad
c^\top = \begin{bmatrix}
           54  &  44  &  15  &  -8 &  -70
         \end{bmatrix},\\
& u^\top = \begin{bmatrix}
           62  &  48   & 36  &  84  &  59
         \end{bmatrix}.
\end{aligned}
\]
In Figure~\ref{fig:bad-behave-of-phi} we plot $\phi(\lambda)$ for $\lambda\in [-8.36,7.00]$. We distinct three cases in \eqref{eqn:phi-type-I}, \eqref{eqn:phi-type-II} and \eqref{eqn:phi-type-III} with three colors blue, red and green, respectively.
As it can be seen in the figure, $\phi(\lambda)$ consists of $4$ pieces.

  \begin{figure}
  \centering
    \includegraphics[width=8cm]{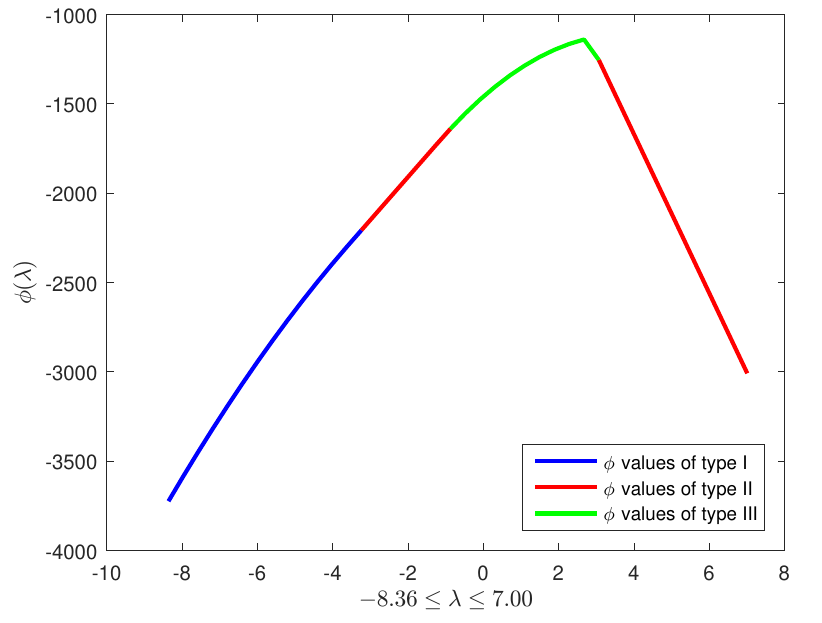}
    \caption{Plot of $\phi(\lambda)$   for the problem of Example~\ref{exm:phi-bad-behavior}.}
    \label{fig:bad-behave-of-phi}
  \end{figure}
\end{exm}
Inner optimization problem in \eqref{eqn:lag-dual} is a special case of problem \eqref{eqn:r1-QKP-unres}, that can be solved by Theorem~\ref{thm:main-result}.
In Theorem~\ref{thm:main-result} it is assumed that coefficients of the linear term  in the objective function is sorted in decreasing order.
In problem  \eqref{eqn:lag-dual}, the order of coefficients of the linear term  depends on $\lambda$.
From now on, we denote by $c_i(\lambda)$ the coefficient of $x_i$, i.e. $c_i(\lambda) = c_i + \lambda a_i$.
Moreover, we denote the line $\{c_i(\lambda)\;:\; \lambda\in\RR\}$ by $\ell_i$.
It is easy to see that, when  $\lambda$ becomes  greater than the intersection of $\ell_i$ and $\ell_j$, $c_i(\lambda)$ and $c_j(\lambda)$ change position in the ordered list  of coefficients.

We use a modification of the well-known plane sweep algorithm to find the ordered intersection points of lines $\{\ell_i \;:\; i=1,\cdots, n\}$.
Now, let $\lambda_a$, $\lambda'$ and $\lambda_c$ are three consecutive intersection points, then because of the Lagrangian function is unimodal, the optimal Lagrange multiplier $\lambda^\ast$ lies in the interval $[\lambda_a,\lambda_b]$ if  $\phi(\lambda') > \phi(\lambda_a)$ and $\phi(\lambda') > \phi(\lambda_b)$.

We modify the implementation of the line-sweep algorithm proposed in \cite{berg2008computational}.
In this algorithm, a vertical line $\ell$ sweeps the plane from left to right.
The \emph{status} of the sweep line is the ordered sequence of lines that intersect it.
The status initially contains all lines in the order of decreasing slope, that is the order of lines when they intersect with sweep line at $\lambda=-\infty$.
The status is updated when $\ell$ reaches an intersecting point.
For example, suppose that the sequence of four lines $\ell_l$, $\ell_i$, $\ell_j$, and $\ell_m$ appears in the status when $\ell$ reaches the intersection point of $\ell_i$ and $\ell_j$.
Then, $\ell_i$ and $\ell_j$ switch position and the intersection of lines $\ell_i$ and $\ell_m$ and the intersection of $\ell_j$ and $\ell_l$ are to be checked.
The new detected intersection points are stored to proceed.
 The order of cost coefficient of the linear term in $\phi(\lambda)$  is unchanged between two consecutive intersection points.

If $c_i(\lambda) < 0$ for some $i$, then $x_i=0$ in the optimal solution of the $\phi(\lambda)$ subproblem.
We introduce a set $Z$ to store the non-vanished variables.
To do so we add a dummy line $\ell_0 : c_0(\lambda) = 0$.
In each intersection of the dummy line and the other lines, the set $Z$ should be updated.
In fact, if $\ell_i$ intersect $\ell_0$ and $i\not\in Z$, then we add $i$ to $Z$, otherwise, if $i\in Z$, then it should be removed from $Z$.
In other words, since we sweep the plane from left to right, if $\ell_i$ intersect $\ell_0$ and $a_i<0$, then we add $i$ to $Z$.
If $\ell_i$ intersect $\ell_0$ and $a_i>0$ then it means $i$ should be removed from $Z$.
$Z$ initially contains the set of all lines with a positive slope.
With this modification, we guarantee that between two consecutive intersection points the set of zero-valued variables is unchanged.
It should be noted here that lines with equal slopes are sorted based on increasing order of $c_i$s.
We summarize the algorithm in the following Algorithm \ref{alg:solving-R1QKP}. This algorithm is used as the first phase in the main algorithm.

\begin{algorithm}
\caption{A plane sweep algorithm for finding an interval containing the optimal solution of the Lagrangian dual problem.}\label{alg:solving-R1QKP}
\begin{algorithmic}[1]
\STATE{\textbf{Input}: vectors $c$, $a$ and $u$ and scaler $b$.}
\STATE{\textbf{Output}: interval $[\lambda_a,\lambda_b]$ that contains the optimal solution of problem $\max_{\lambda \in\RR} \phi(\lambda)$ or the smallest and largest intersection points.}
\STATE{Initialize state array, $\ell=[1,\cdots,n]$ , with lines $\ell[1],\cdots,\ell[n]$ sorted in decreasing order of slope.}\label{alg:solving-R1QKP-init-start}
\STATE{Initialize queue $Q = \emptyset$.}
\STATE{Initialize line indices array $p = [1,\cdots, n]$.}
\STATE{\textsc{Fail} $\leftarrow$ \verb|true| }
\STATE{Insert intersection points  of all adjacent lines into $Q$.}\label{alg:solving-R1QKP-init-end}
\STATE{Set $\lambda^{\prev} \leftarrow -\infty$, $\lambda^{\prev\prev} \leftarrow -\infty$}
\WHILE{$Q$ is not empty}
\STATE{Pop from $Q$ the current intersection point $\lambda^{\new}$ and the corresponding two adjacent lines $\ell[i]$ and $\ell[j]$.}
\STATE{Update state array: $\ell[p[i]] \leftrightarrow \ell[p[j]]$.}\label{alg:solving-R1QKP-updating-ell}
\STATE{Update the line indices array: $p[i]\leftrightarrow p[j]$.}
\STATE{Insert the intersection point of $\ell[p[i]]$ and $\ell[p[i]+1]$  and the intersection point of $\ell[p[j]]$ and $\ell[p[j]-1]$ into $Q$, if there exists any.}
\IF{$\phi(\lambda^{\prev})>\phi(\lambda^{\prev\prev})$ and $\phi(\lambda^{\prev})> \phi(\lambda^{\new})$ }\label{alg:solving-R1QKP-subproblem}
\STATE{Set \textsc{Fail} $\leftarrow$ \verb|false|}
\RETURN{$[\lambda^{\prev\prev},\lambda^{\new}]$ as the promissing interval.}
\ENDIF
\STATE{Set $\lambda^{\prev\prev} \leftarrow \lambda^{\prev}$.}
\STATE{Set $\lambda^{\prev} \leftarrow \lambda^{\new}$.}
\ENDWHILE
\IF{\textsc{Fail}}
\RETURN{the smallest $\lambda_{LB}$ and the largest $\lambda_{UB}$ interseation points.}
\ENDIF
\end{algorithmic}
\end{algorithm}

\begin{thm}\label{thm:algorithm1-complexity}
  Algorithm~\ref{alg:solving-R1QKP} runs in  $O(n^2\log n)$ time.
\end{thm}

\begin{proof}
  Initializing state array  $\ell$, line indices array $p$ and the queue $Q$ in steps~\ref{alg:solving-R1QKP-init-start}--\ref{alg:solving-R1QKP-init-end} needs $O(n\log n)$ time.
  In each iteration we perform two main operations: computing the value of $\phi$ for a new intersect point $\lambda^{\new}$ and updating $Q$.
  The order of $c_i(\lambda^{\new})$ and the vector $G$  can be updated from the previous intersection point in $O(1)$ time.
  Finding $\bar n$ needs $O(\log n)$, using binary search.
  On the other hand, insertion and deletion on the priority queue  $Q$ takes $O(\log n)$ since one can implement the priority queue by a heap to store the intersection points.
  Therefore each iteration of the main loop needs $O(\log n)$ time. Since there are $O(n^2)$ intersection points, the algorithm runs in $O(n^2\log n)$.
\end{proof}

Let $\lambda_{LB}$ and $\lambda_{UB}$ be the smallest and greatest intersection points of lines $\{ \ell_i\;:\; i=1,\cdots, n\}$.
The optimal solution of the Lagrangian problem may lie out of the interval $[\lambda_{LB},\lambda_{UB}]$.
In this case, Algorithm~\ref{alg:solving-R1QKP} fails to find the optimal interval.
So, we explore the outside of  $[\lambda_{LB},\lambda_{UB}]$ in a separate phase.

First consider the exploration of $(-\infty, \lambda_{LB})$.
Since the components of vector $G$ in Theorem~\ref{thm:main-result} are linear functions in term of $\lambda$, then there exists $\lambda'_{LB}<\lambda_{LB}$  such that the order of $G_k$s does not change for $\lambda<\lambda'_{LB}$.
Indeed, a similar procedure for finding the smallest intersection of lines $\ell_i$'s can be used here to compute $\lambda'_{LB}$.
Now, since $\phi(\lambda)$ is unimodal one can conclude that
\begin{equation}\label{eqn:phaseII-prob}
  \max_{(-\infty, \lambda_{LB}]} \phi(\lambda) = \max_{[\lambda'_{LB}, \lambda_{LB}]} \phi(\lambda).
\end{equation}
Similarly, for the values of $\lambda$ greater than $\lambda_{UB}$, one can find a threshold, say $\lambda'_{UB}$, such that
\begin{equation}\label{eqn:phaseIII-prob}
\max_{[\lambda_{UB},\infty)} \phi(\lambda) = \max_{[\lambda_{UB}, \lambda'_{UB}]} \phi(\lambda).
\end{equation}
We summarize the whole algorithm in Algorithm~\ref{alg:solving-general-R1QKP}.

\begin{algorithm}
\caption{A two-phase algorithm for solving rank-one quadratic knapsack problem \eqref{eqn:r1-qkp-red}.}\label{alg:solving-general-R1QKP}
\begin{algorithmic}[1]
\STATE{Run Algorithm~\ref{alg:solving-R1QKP} to find a promising interval that contains the optimal Lagrange multiplier.}
\IF{  Algorithm~\ref{alg:solving-R1QKP} returns an interval $[\lambda_a,\lambda_b]$ }
\STATE{
    Solve optimization problem $\max_{[\lambda_a, \lambda_b]} \phi(\lambda)$ and return the optimal solution.
    }\label{alg:step5}
\ELSE
\STATE{Solve  optimization problems \eqref{eqn:phaseII-prob} and \eqref{eqn:phaseIII-prob} and store the optimal values.}\label{alg:solving-general-R1QKP-phaseII}
\ENDIF
\RETURN{the best $\lambda$ found as an optimal Lagrange multiplier.}
\end{algorithmic}
\end{algorithm}
It is clear that Algorithm~\ref{alg:solving-general-R1QKP} converges to the optimal solution, since the output interval of Algorithm~\ref{alg:solving-R1QKP} contains the optimal solution and $\phi(\lambda)$ is unimodal. In fact, the single variable optimization problem in step~\ref{alg:step5} can be solved efficiently by a classical root-finding algorithms.

\section{Computational Experiments}\label{sec:comp-test}
In this section, we compare the running time and the numerical accuracy of Algorithm~\ref{alg:solving-general-R1QKP} with the general convex quadratic programming solver, CPLEX.
We implement Algorithm~\ref{alg:solving-general-R1QKP} with  MATLAB R2019b. All runs are performed on a system with Core i7 $2.00$ GHz CPU and $8.00$ GB of RAM equipped with a $64$bit Windows operating system.
We solve single variables optimization problems \eqref{eqn:phaseIII-prob}, \eqref{eqn:phaseII-prob} and step~\ref{alg:step5} in Algorithm~\ref{alg:solving-general-R1QKP}, using MATLAB built-in function \textsf{fminbnd} which is based on  golden section search and parabolic interpolation.

Our testbed contains two types of randomly generated rank-one knapsack problems up to $n=50,000$ variables.
In the first type, vectors $a$ and $c$ are integral and generated uniformly from the same interval. We denote this type by \textsf{TypeI}.
In the second type (\textsf{TypeII}),  vectors $a$ and $c$ are positive and negative randomly generated integral vectors, respectively.
In Table~\ref{tbl:param} we summarize the parameter values for problem instances.

\setlength{\heavyrulewidth}{0.1em}
\setlength{\extrarowheight}{-1em}

\begin{table}
\caption{Parameters for two types of problem instances.}\label{tbl:param}
\centering
{
\begin{tabular}{l*{4}{l}}\toprule
Type  &  $a$  & $c$ & $l$ & $u-l$  \ML
\textsf{TypeI}	& $U(-50,50)$	& $U(-50,50)$	& $U(0,20)$	& $U(1,100)$\\
\textsf{TypeII}	& $U(-100,10)$	& $U(10,100)$	& $U(0,20)$	& $U(1,100)$\LL
\end{tabular}
}
\end{table}

As a well-known general convex quadratic programming solver, we chose CPLEX (ver. 12.9)  to compare with our results.
Based on our numerical tests,  we set the quadratic programming solver of CPLEX (\textsf{qpmethod} option) to \textsf{barrier}. The \textsf{barrier} convergence tolerance, \textsf{convergetol} is set to $1e-12$.
After we complete our experimental tests, we found in \cite{cplextrick} that the sparsity of the Hessian matrix influences the performance of CPLEX. To increase the performance, we reformulate our problem as
\[
\min\left\{ \frac{1}{2} y^2 - c^\top x \;:\; \bone^\top x - y =0, \; a^\top x = b,\; 0\leq x\leq u\right\}.
\]
We denote the results corresponding to running CPLEX on the original problem and the above modification respectively by $\text{CPLEX}_{org}$ and $\text{CPLEX}_{ref}$.

Table~\ref{tbl:result} shows the average running time for $5$ runs of each algorithm/solver for each problem size. Dash sign, '-', denoted the algorithm/solver encounters out-of-memory status.

In all cases, Algorithm~\ref{alg:solving-general-R1QKP} outperforms $\text{CPLEX}_{org}$. For instances up to $n=5000$, our algorithm and $\text{CPLEX}_{ref}$ has same performance, whereas for larger instances, $\text{CPLEX}_{ref}$ has smaller running time.

\begin{table}
\caption{A comparison of running times (in seconds) between our algorithm and $\text{CPLEX}_{org}$ and $\text{CPLEX}_{ref}$.}\label{tbl:result}
\centering
{
\begin{tabular}{l*{5}{c}}\toprule
$n$   & &  Our algorithm  & $\text{CPLEX}_{org}$ & $\text{CPLEX}_{ref}$ \ML
$1000$ & 	 \textsf{TypeI} & 	 $0.06$ & 	 $0.09 $ & 	 $0.01$ \\
       & 	 \textsf{TypeII} & 	 $0.01$ & 	 $0.06 $ & 	 $0.02$ \\
$1500$ & 	 \textsf{TypeI} & 	 $0.04$ & 	 $0.15 $ & 	 $0.02$ \\
       & 	 \textsf{TypeII} & 	 $0.02$ & 	 $0.13 $ & 	 $0.02$ \\
$2000$ & 	 \textsf{TypeI} & 	 $0.04$ & 	 $0.27 $ & 	 $0.02$ \\
       & 	 \textsf{TypeII} & 	 $0.02$ & 	 $0.27 $ & 	 $0.02$ \\
$5000$ & 	 \textsf{TypeI} & 	 $0.09$ & 	 $2.21 $ & 	 $0.02$ \\
       & 	 \textsf{TypeII} & 	 $0.06$ & 	 $2.12 $ & 	 $0.03$ \\
$10000$ & 	 \textsf{TypeI} & 	 $0.26$ & 	 $16.26 $ & 	 $0.04$ \\
       & 	 \textsf{TypeII} & 	 $0.23$ & 	 $16.95 $ & 	 $0.05$ \\
$15000$ & 	 \textsf{TypeI} & 	 $0.62$ & 	 $61.20 $ & 	 $0.10$ \\
       & 	 \textsf{TypeII} & 	 $0.63$ & 	 $65.88 $ & 	 $0.10$ \\
$20000$ & 	 \textsf{TypeI} & 	 $1.16$ & 	 - & 	 $1.20$ \\
       & 	 \textsf{TypeII} & 	 $0.88$ & 	 - & 	 $1.02$ \\
$50000$ & 	 \textsf{TypeI} & 	 $3.22$ & 	 - & 	 $0.11$ \\
       & 	 \textsf{TypeII} & 	 $3.19$ & 	 - & 	 $0.11$ \\
$100000$ & 	 \textsf{TypeI} & 	 $12.19$ & 	 - & 	 $0.14$ \\
        & 	 \textsf{TypeII} & 	 $11.31$ & 	 - & 	 $0.17$ \ML
\end{tabular}
}
\end{table}

%

\section{Conclusions}
In this paper, we proposed a two-phase algorithm for rank-one quadratic knapsack problems. To this end, we studied the solution structure of the problem when it has no resource constraint.
Indeed, we proposed an $O(n\log n)$ algorithm to solve this problem.
We then used the solution structure to propose an $O(n^2\log n)$ line-sweep algorithm that finds an interval that contains the optimal Lagrange multiplier.
Then, the estimated optimal interval is explored for computing the optimal solution with the desired accuracy.
Our computational tests showed that our algorithm has better performance than CPLEX when CPLEX is used to solve the original problem. After a reformulation of the problem, CPLEX outperforms our algorithm for instances with $n\geq 5000$.

\clearpage
%

\end{document}